\documentclass[12pt,a4paper,reqno]{amsart}

\usepackage[latin1]{inputenc}
\usepackage{mathptmx}
\usepackage{amssymb}
\usepackage{hyperref}

\def\ZZ{{\mathbb Z}}

\def\RR{{\mathbb R}}
\def\QQ{{\mathbb Q}}

\def\rank{\operatorname{rank}}
\def\cone{\operatorname{cone}}
\def\conv{\operatorname{conv}}
\def\gp{\operatorname{gp}}
\def\ed{\operatorname{ed}}
\def\vmult{\operatorname{vmult}}

\newtheorem{theorem}{Theorem}

\newtheorem{proposition}[theorem]{Proposition}

\theoremstyle{definition}
\newtheorem{remark}[theorem]{Remark}

\numberwithin{equation}{section}

\begin{document}

\title{The computation of generalized Ehrhart series in Normaliz}
\author[W. Bruns]{Winfried Bruns}
\address{Winfried Bruns\\ Universit\"at Osnabr\"uck\\ FB Mathematik/Informatik\\ 49069 Osna\-br\"uck\\ Germany}
\email{wbruns@uos.de}

\author[C. S\"oger]{Christof S\"oger}
\address{Christof S\"oger\\ Universit\"at Osnabr\"uck\\ FB Mathematik/Informatik\\ 49069 Osna\-br\"uck\\ Germany}
\email{csoeger@uos.de}

\subjclass[2010]{52B20, 13F20, 14M25, 91B12}

\begin{abstract}
we describe an algorithm for the computation of generalized (or
weighted) Ehrhart series based on Stanley decompositions as
implemented in the offspring NmzIntegrate of Normaliz. The
algorithmic approach includes elementary proofs of the basic
results. we illustrate the computations by examples from
combinatorial voting theory.
\end{abstract}

\maketitle

Let $M\subset \ZZ^n$ be an affine monoid endowed with a
positive $\ZZ$-grading $\deg$. Then the \emph{Ehrhart} or
\emph{Hilbert series} is the generating function
$$
E_M(t)=\sum_{x\in M} t^{\deg x}=\sum_{k=0}^\infty \#\{x\in M:\deg x=k\}t^k,
$$
and $E(M,k)=\#\{x\in M:\deg x=k\}$ is the Ehrhart or Hilbert
function of $M$ (see \cite{BG} for terminology and basic
theory). It is a classical theorem that $E_M(t)$ is the power
series expansion of a rational function of negative degree at
$t_0=0$ and that $E(M,k)$ is given by a quasipolynomial of
degree $\rank M-1$ with constant leading coefficient equal to
the (suitably normed) volume of the rational polytope
$$
P=\cone(M)\cap A_1
$$
where $\cone(M)\subset \RR^n$  is the cone generated by $M$ and
$A_1$ is the hyperplane of degree $1$ points. In the following
we assume that
$$
M=\cone(M)\cap L
$$
for a sublattice $L$ of $\ZZ^n$. Then $E(M,k)$ counts the
$L$-points in the multiple $kP$, and is therefore the Ehrhart
function of $P$ (with respect to $L$).

Monoids of the type just introduced are important for
applications, and in some of them, like those discussed in
Section \ref{Compute}, one is naturally led to consider
\emph{generalized} (or weighted) \emph{Ehrhart series}
$$
E_{M,f}(t)=\sum_{x\in M}f(x)t^{\deg x}
$$
where $f$ is a polynomial in $n$ indeterminates. It is
well-known that also the generalized Ehrhart series is the
power series expansion of a rational function; see \cite{BB1},
\cite{BB2}.

In the last months we have implemented an offspring of Normaliz
\cite{Nmz} called  NmzIntegrate\footnote{NmzIntegrate will be
uploaded to \cite{Nmz} together with Normaliz 2.9 by February
2013.} that computes generalized Ehrhart series. The input
polynomials of NmzIntegrate must have rational coefficients,
and we assume that $f$ is of this type although it is
mathematically irrelevant. This note describes the computation
of generalized Ehrhart series based on Stanley decompositions
\cite{Sta1}. Apart from taking the existence of Stanley
decompositions as granted, we give complete and very elementary
proofs of the basic facts. They follow exactly the
implementation in NmzIntegrate (or vice versa).

The generalized Ehrhart function is given by a quasipolynomial
$q(k)$ of degree $\le \deg f+\rank M-1$, and the coefficient of
$k^{\deg f+\rank M-1}$ in $q(k)$ can easily be described as the
integral of the highest homogeneous component of $f$ over the
polytope $P$. Therefore we have also included (and implemented)
an approach to the computation of integrals of polynomials over
rational polytopes in the spirit of the Ehrhart series
computation. See \cite{BB2} and \cite{Latte} for more
sophisticated approaches.

\emph{Acknowledgement.}\enspace We gratefully acknowledge the
support we received from John Abbott and Anna Bigatti in using
CoCoALib \cite{CoC}, on which the multivariate polynomial
algebra in NmzIntegrate is based.

\section{The computation of generalized Ehrhart series}

Via a Stanley decomposition and substitution the computation of
generalized Ehrhart series can be reduced to the case in which
$M$ is a free monoid, and for free monoids one can split off
the variables of $f$ successively so that one ends at the
classical case $M=\ZZ_+$. We take the opposite direction,
starting from $\ZZ_+$.

\subsection{The monoid $\ZZ_+$} Let $M=\ZZ_+$. By linearity it is enough to consider the polynomials $f(k)=k^m$, $k\in\ZZ_+$, for which the generalized Ehrhart series is given by
$$
\sum_{k=0}^\infty k^m t^{um}, \qquad u=\deg 1,
$$
and if necessary we can assume $u=1$, substituting $t\mapsto
t^u$ in the final result.

The rising factorials
$$
(k+1)_m=(k+1)\cdots(k+m)
$$
form a $\ZZ$-basis of the polynomial ring $\ZZ[k]$. Therefore
we can write
\begin{equation}\label{Stirling}
k^m=\sum_{j=0}^m s_{m,j}(k+1)_j
\end{equation}
and use that
\begin{equation}\label{risfact}
\sum_{k=0}^\infty (k+1)_r\,t^k = \sum_{j=r}^\infty (t^{j})^{(r)} =
\sum_{j=0}^\infty (t^{j})^{(r)} = \biggl(\frac1{1-t}\biggr)^{(r)}=\frac{r!}{(1-t)^{r+1}}.
\end{equation}
Equations \eqref{Stirling} and \eqref{risfact} solve our
problem for $M=\ZZ_+$ and $f(k)=k^m$:
\begin{equation}\label{onevar}
\sum_{k=0}^\infty k^m t^{uk}=\frac{A_{m,u}(t)}{(1-t^u)^{m+1}},\qquad A_{m,u}(t)\in\ZZ[t].
\end{equation}
It is enough to compute $A_{m,1}(t)$ because
$A_{m,u}(t)=A_{m,1}(t^u)$. One should note that $A_{m,u}$ is a
polynomial of degree $m$. Therefore the rational function in
\eqref{onevar} has negative degree.

Since the coefficient $s_{m,m}$ of $(k+1)_m$ in the
representation of $k^m$ is evidently equal to~$1$, we have
\begin{equation}\label{mult1}
\sum_{k=0}^\infty k^m t^{um}=\frac{m!}{(1-t)^{m+1}}+
\textup{terms of smaller pole order at $t=1$}
\end{equation}

\begin{remark}\label{classical}
The coefficients $s_{m,j}$ in \eqref{Stirling} and the
coefficients of the polynomials $A_{m,1}$ are well-known
combinatorial numbers.

(a)  $s_{m,j}=(-1)^{m-j}S(m+1,j+1)$ where $S(p,q)$ is the
Stirling number of the second kind that counts the number of
partitions of a $p$-set into $q$ blocks. This follows
immediately from the classical identity
$k^{m+1}=\sum_{j=1}^{m+1}(-1)^{m+1-j}S(m+1,j)(k)_j$ (for
example, see Stanley \cite[4.3,c]{Sta}).

(b) For $m=0$ we have $A_{0,1}=1$ and $A_{m,1}=\sum_{j=1}^m
A(m,j)t^j$ for $m>0$ where $A(m,j)$ is the Eulerian number
\cite[4.3,d]{Sta}.
\end{remark}

\subsection{The monoid $\ZZ_+^d$}
Next we consider $M=\ZZ_+^d$. The crucial observation is that
the problem is multiplicative for products of polynomials in
disjoint variables. Suppose that $f(x)=g(y)h(z)$,
$y=(x_1,\dots,x_r)$, $z=(x_{r+1},\dots,x_d)$. Then
\begin{equation}\label{product}
E_{M,f}(t)=\sum_{x\in \ZZ_+^d}f(x)t^{\deg x}=\biggl(\sum_{y\in
\ZZ_+^r}g(y)t^{\deg y}\biggr)\biggl(\sum_{z\in \ZZ_+^{d-r}}h(z)t^{\deg z}\biggr)
\end{equation}
by multiplication of power series.

In order to exploit \eqref{product} we split the last variable
off,
$$
f(x)=\sum_i f_i(x_1,\dots,x_{d-1})x_d^i,
$$
and obtain
\begin{align}
E_{M,f}(t)&=\sum_i\left(\biggl(\sum_{x'\in \ZZ_+^{d-1}}f_i(x')t^{\deg x'}\biggr)
\biggl(\sum_{k=0}^\infty k^it^{ui}\biggr)\right)\nonumber\\
&=\sum_i\left( \frac{A_{i,u}(t)}{(1-t^u)^{i+1}}\sum_{x'\in \ZZ_+^{d-1}}f_i(x')t^{\deg x'}\right)
\label{reduct}
\end{align}
with $u=\deg e_d$.

Applying this formula inductively allows us to eliminate all
variables $x_i$ and to end with the desired representation of
$E_{\ZZ_+^d,f}(t)$.

Generalizing \eqref{mult1}, let us consider the case in which
$f$ is a monomial, $f(x_1,\dots,x_d)=x_1^{m_1}\cdots
x_d^{m_d}$, and $\ZZ_+^d$ is endowed with its \emph{standard
degree}, $\deg(x)=x_1+\dots+x_d$. Then equations
\eqref{product} and \eqref{mult1} imply that
\begin{equation}\label{mult2}
E_{M,f}(t)=\frac{m_1!\cdots m_d!}{(1-t)^{m_1+\dots+m_d+d}}+
\textup{terms of smaller pole order at $t=1$}.
\end{equation}

\subsection{Using the Stanley decomposition}

We now turn to general $M \subset \ZZ^n$. Normaliz computes a
triangulation $\Sigma$ of $\cone(M)$ into simplicial subcones
$\sigma$. Moreover, it computes a \emph{disjoint} decomposition
$$
\cone(M)=\bigcup_{\sigma\in\Sigma} \sigma\setminus S_\sigma
$$
where $S_\sigma$ is a union of facets of $\sigma$. The
existence of such a decomposition is a nontrivial fact.
Classically it is derived from the Brugesser-Mani theorem on
the existence of line shellings (see Stanley \cite{Sta1}).
Instead of a line shelling, Normaliz (now) uses a method of
Köppe and Verdoolaege: see \cite{KV} and \cite[Section 4]{BIS}.

Every simplicial subcone (of full dimension) is generated by
linearly independent vectors $v_1,\dots,v_d\in M$, $d=\rank M$.
They generate a free submonoid $M_\sigma$ of $M$. For
every~$\sigma$ Normaliz computes the set
$$
E_\sigma=\bigl\{x\in\gp(M): x=\alpha_1v_1+\dots+\alpha_dv_d,\ \alpha_i\in [0,1)\bigr\}.
$$
For $x\in E_\sigma$ we let $\epsilon(x)$ be the sum of those
$v_i$ for which (i) $\alpha_i=0$ and (ii) the facet of $\sigma$
opposite to $v_i$ lies in the excluded set $S_\sigma$ (so that
$x$ lies in the excluded set). Then it is not hard to see that
we have a disjoint decomposition
$$
M=\bigcup_{\sigma\in\Sigma}\  \bigcup_{x\in E_\sigma} x+\epsilon(x)+M_\sigma.
$$
It is called a \emph{Stanley decomposition} since its existence
is originally due to Stanley \cite{Sta1}.

In the following we set $\widetilde x=x+\epsilon(x)$ and
$$
N_{\sigma,x}=\widetilde x+M_\sigma.
$$
Then
$$
E_{M,f}(t)=\sum_\sigma\sum_{x\in E_\sigma} E_{N_{\sigma,x},f}(t).
$$
Set $d=\rank M$, and for given $\sigma$ consider the linear map
$$
\alpha_\sigma:\ZZ_+^d \to \ZZ^n,\qquad \alpha_\sigma(y_1,\dots,y_d)=y_1v_1+\dots+\dots+y_dv_d,
$$
where $v_1,\dots,v_d$ is the generating set of $M_\sigma$ as
above. With
\begin{align}
\deg_\sigma y=\deg \alpha_\sigma(y),\nonumber\\
g_{\sigma,x}(y)=f\bigl(\alpha_\sigma(y)+\widetilde x\bigr),\label{transform_eq}
\end{align}
we have
$$
E_{N_{\sigma,x},f}(t)=t^{\deg \widetilde x}\sum_{y\in\ZZ_+^d}g_{\sigma,x}(y)t^{\deg_\sigma y}.
$$
This equation transforms the summation over $N_{\sigma,x}$ into
a summation over $\ZZ_+^d$. Then we can apply \eqref{reduct}
inductively to
\begin{equation}\label{sigma}
\widetilde E_{\sigma,f}(t)=\sum_{x\in E_\sigma} E_{N_{\sigma,x},f}(t).
\end{equation}
Finally, we sum the rational functions $\widetilde
E_{\sigma,f}(t)$ over the triangulation $\Sigma$.

\begin{remark}
(a) Instead of applying \eqref{reduct} to every $\sigma$, we
accumulate the polynomials $g_{\sigma,x}$ over all $\sigma$
that induce the same degree $\deg_\sigma$ on $\ZZ^d$ (the
classes formed in this way are called \emph{denominator
classes}).

(b) The time critical steps in the algorithm are
\begin{enumerate}
\item the coordinate transformation \eqref{transform_eq},
    and
\item the inductive application of \eqref{reduct}.
\end{enumerate}
In order to speed up (1), we factor the polynomial $f$,
transform the factors separately, and multiply the transformed
factors. If $f$ happens to decompose into linear factors, then
multiplication of linear polynomials becomes a time critical
step. In order to speed up (2) we have introduced the
denominator classes.

(c) Note that $\sum_{y\in\ZZ_+^d}g_{\sigma,x}(y)t^{\deg_\sigma
y}$ is invariant under permutations of variables $y_i$ that
preserve the degrees $\deg_\sigma e_i$. Therefore one can go
over $g_{\sigma,x}$ monomial by monomial and reorder the
exponent vectors in such a way that the exponents of variables
corresponding to the same degree become decreasing. The
reordering significantly  reduces the number of monomials in
the polynomials to which \eqref{reduct} must be applied, saves
memory and also speeds up \eqref{reduct}.

(d) We want to point out that \eqref{reduct} is \emph{not}
applied recursively. Instead the right hand side is expanded
after the elimination of $x_d$, and $x_{d-1}$ is then
eliminated from the resulting polynomial whose coefficients are
rational functions in $t$. This procedure is repeated until all
$x_i$ have been eliminated.

\end{remark}

\section{The quasipolynomial, its virtual leading coefficient,
and integration}

\subsection{The quasipolynomial} All rational functions in $t$
that come up in \eqref{sigma} can be written over the
denominator
$$
(1-t^\ell)^{\deg f+\rank M}
$$
where $\ell$ is the least common multiple of the numbers $\deg
x$ for the generators $x$ of $M$ that appear in the
triangulation. This follows from \eqref{reduct} if one observes
that $1-t^u$ divides $1-t^\ell$. Moreover, all summands have
negative degree as rational functions in $t$. Therefore
\cite[4.4.1]{Sta} implies the following proposition.

\begin{proposition}\label{quasi}
$$
E_{M,f}(t)=\sum_{k=0}^\infty q(k)t^k
$$
where $q$ is a rational quasipolynomial of period $\pi$
dividing $\ell$ and of degree $\le\deg f+\rank M-1$.
\end{proposition}

The statement about the quasipolynomial means that there exist
polynomials $q^{(j)}$, $j=0,\dots,\pi-1$, of degree $\le\deg
f+\rank M-1$ such that
$$
q(k)=q^{(j)}(k),\qquad j\equiv k\pod \pi,
$$
and
$$
q^{(j)}(k)=q^{(j)}_0+q^{(j)}_1k+\dots+q^{(j)}_{\deg f+\rank M-1}k^{\deg f+\rank M-1}
$$
with coefficients $q^{(j)}_i\in \QQ$. As we will see below, it
is justified to call
$$
\ed(M,f)=\deg f+\rank M-1
$$
the \emph{expected degree} of $q$.

\subsection{The virtual leading coefficient and Lebesgue
integration}

Let $m=\deg f$ and write $f=f_m+g$ where $f_m$ is the degree
$m$ homogeneous component of $m$. Then $\deg g<m$, and it
follows from Proposition \ref{quasi} that $g$ does not
contribute to the coefficient $q^{(j)}_{\ed(M,f)}$. Moreover,
this coefficient is independent of $j$ and given by an
integral, as we will see in Proposition \ref{inte} below.

For the representation as an integral we must norm the measure
in such a way that it is compatible with the lattice structure.
We will integrate over the polytope
$$
P=\cone(M)\cap A_1,\qquad A_1=\{x\in\RR^n:\deg x=1\}.
$$
Let $L_0=L\cap \RR M\cap A_0$ where $A_0=\{x\in\RR^n:\deg
x=0\}$ is the linear subspace of degree $0$ elements. Then
$L_0$ is a (saturated) sublattice of $L$ of rank $d-1$
($d=\rank M$), and we choose a basis $u_1,\dots,u_{d-1}$ of
$L_0$. Note that $H=\RR M\cap A_1$ has dimension $d-1$ and
contains a point $z\in L$ since we have required that $\deg$
takes the value $1$ on $\gp(M)$, and we can consider the
\emph{basic $L_0$-simplex}
$\delta=\conv(z,z+u_1,\dots,z+u_{d-1})$ in $H$. Now we norm the
Lebesgue measure $\lambda$ on $H$ by giving  volume $1/(d-1)!$
to the basic $L_0$-simplex. (The measure is independent of the
choice of $\delta$ since two basic $L_0$-simplices differ by an
affine-integral automorphism of $H$.) We call $\lambda$ the
\emph{$L$-Lebesgue measure} on $H$.

\begin{proposition}\label{inte}
For all $j=0,\dots,\pi-1$ one has
\begin{equation}\label{lcinte}
q^{(j)}_{\ed(M,f)}=\int_P f_m\,d\lambda.
\end{equation}
\end{proposition}

\begin{proof}
We may assume that $f$ is homogeneous of degree $m$. Let
$$
L_c=\frac 1c L.
$$
Then
$$
\int_P f_m\,d\lambda=\lim_{c\to \infty} \sum_{x\in P\cap L_c} \frac1{c^{d-1}}f(x)
$$
by elementary integration theory.

Note that
$$
f(x)=\frac 1{c^m}f(cx)
$$
by homogeneity and that $x\in P\cap L_c$ if and only $cx\in
L\cap cP$. Thus
$$
\int_P f_m\,d\lambda=\lim_{c\to \infty} \sum_{y\in cP\cap L} \frac1{c^{m+d-1}}f(y).
$$
On the other hand, we obtain $q^{(j)}_{\ed(M,f)}$ as the limit
over the subsequence $(b\pi +j)_{b\in\ZZ_+}$:
$$
q^{(j)}_{\ed(M,f)}=\lim_{b\to\infty} \sum_{y\in (b\pi+j)P\cap L} \frac1{(b\pi+j)^{m+d-1}} f(y)
$$
by Proposition \ref{quasi}. This concludes the proof.
\end{proof}

In view of Proposition \ref{inte} it is justified to call
$q_{\ed(M,f)}=q^{(j)}_{\ed(M,f)}$ the \emph{virtual leading
coefficient}, and the proposition justifies the term ``expected
degree'' for $\deg f+\rank M-1$ the. In analogy with the
definition of multiplicity in commutative algebra (for example,
see \cite{BH}), we call
$$
\vmult(M,f)=\ed(M,f)!q_{\ed(M,f)}
$$
the \emph{virtual multiplicity of $(M,f)$}. It is an integer if
$P$ is a lattice polytope and $f_m$ has integral coefficients,
as we will see below.

\subsection{Computing the integral}

It is natural to compute the integral by summation over the
triangulation: the triangulation of $\cone(M)$ into simplicial
subcones $\sigma$ induces a triangulation of the polytope $P$
into simplices $\delta=\sigma\cap P$. As usual let
$v_1,\dots,v_d\in M$ be the generators of $\sigma$. Then
$\delta$ is spanned by the degree $1$ vectors $v_i/\deg(v_i)$,
$i=1,\dots,n$. Let $e_1,\dots,e_d$ be the unit vectors in
$\RR^d$. Then the substitution $e_i\mapsto v_i/\deg(v_i)$
induces a linear map $\RR^d\to \RR M$ that in its turn
restricts to an affine map $\alpha$ from the standard degree
$1$ hyperplane in $\RR^d$ spanned by $e_1,\dots,e_d$ to the
hyperplane $H = A_1\cap \RR M$, and the image of the unit
simplex $\Delta$ is just $\delta$.

\begin{proposition}\label{transform}
One has
\begin{equation}\label{inttrans}
\int_\delta f\,d\lambda =\frac{|\det_L(v_1,\dots,v_d)|}{\deg(v_1)\cdots\deg(v_d)}
\int_\Delta (f\circ \alpha)\, d\mu
\end{equation}
where $\mu$ is the $\ZZ^d$-Lebesgue measure on the hyperplane
$\widetilde H$ of standard degree $1$ in $\RR^d$ and
$\det_L(v_1,\dots,v_d)$ is the determinant of the coefficient
matrix of $v_1,\dots,v_d$ with respect to a basis of $L\cap \RR
M$.
\end{proposition}

\begin{proof}
This is just the substitution rule if one observes that the
absolute value of the functional determinant of
$\alpha|\widetilde H$ is given by the factor in front of the
integral. For an affine map the functional determinant is
constant. So we can assume $f=1$ and it remains to relate the
volumes of $\delta$ and $\Delta$. But $\Delta$ has volume
$1/(d-1)!$ with respect to $\mu$ and $\delta$ has volume
$$
\frac 1{(d-1)!}\frac{|\det_L(v_1,\dots,v_d)|}{\deg(v_1)\cdots\deg(v_d)}
$$
with respect to $\lambda$; see \cite[Section 4]{BIS}.
\end{proof}

After the substitution it remains to evaluate the integral over
$\Delta$, and this can be done monomial by monomial:

\begin{proposition}\label{monomial}
\begin{equation}\label{intmon}
\int_\Delta y_1^{m_1}\cdots y_d^{m_d}\, d\mu=\frac{m_1!\cdots m_d!}{(m_1+\dots+m_d+d-1)!}.
\end{equation}
\end{proposition}

\begin{proof}
Let $g=y_1^{m_1}\cdots y_d^{m_d}$ and $M=\ZZ_d^+$. Then
$$
E_{M,g}(t)=\frac{m_1!\cdots m_d!}{(1-t)^{(m_1+\dots+m_d+d)}}
+\textup{terms of smaller pole order at $t=1$},
$$
as stated in \eqref{mult2}.

The quasipolynomial is a true polynomial in this case, and the
(virtual) multiplicity is given by the value of the numerator
polynomial at $t=1$, namely $m_1!\cdots m_d!$ (for example, see
\cite[4.1.9]{BH}). Now Proposition \ref{inte} gives the
integral.
\end{proof}

\section{Computational examples}\label{Compute}

We illustrate the use of NmzIntegrate  by three related
examples coming from combinatorial voting theory that are
discussed in \cite{Sch}. We refer the reader to \cite{LLS},
\cite{Sch} or \cite{WP} for a more extensive treatment.

Consider an election in which each of the $k$ voters fixes a
linear preference order of $n$ candidates. In other words,
voter $i$ chooses a linear order of the candidates $1,\dots,n$.
Each such order represents a permutation of $1,\dots,n$. Set
$N=n!$. The result of the election is an $N$-tuple
$(x_1,\dots,x_N)$ in which $x_p$ is the number of voters that
have chosen the preference order labeled $p$. Then
$x_1+\dots+x_N=k$, and $(x_1,\dots,x_N)$ can be considered as a
lattice point in the positive orthant of $\RR_+^N$, or, more
precisely, as a lattice point in the simplex
$$
U_k^{(n)}=\RR_+^N\cap A_k=k\bigl(\RR_+^N\cap A_1\bigr)=kU^{(n)}
$$
where $A_k$ is the hyperplane defined by $x_1+\dots+x_N=k$, and
$U^{(n)}=U_1^{(n)}$ is the unit simplex of dimension $N-1$
naturally embedded in $N$-space. We assume that all lattice
points in the simplex $U_k^{(n)}$ have equal probability of
being the outcome of the election.

The following three problems have been considered in \cite{Sch}
for $4$ candidates $A,B,C,D$:
\begin{enumerate}
\item the Condorcet paradox,
\item the Condorcet efficiency of plurality voting,
\item plurality voting versus cutoff.
\end{enumerate}
For $n=4$ one has $N=24$, and the dimension of the polytope
$U^{(4)}$ is already quite large.

Let us say that candidate $A$ \emph{beats} candidate $B$ if the
number of voters that prefer candidate $A$ to candidate $B$ is
larger than the number of voters with the opposite preference.
Candidate $A$ is the \emph{Condorcet winner} if $A$ beats all
other candidates. As the Marquis de Condorcet noticed, the
relation ``beats'' is nontransitive for some outcomes of the
election, and there may be no Condorcet winner. This phenomenon
is called the \emph{Condorcet paradox}. Problem (1) asks for
its asymptotic probability as the number $k$ of voters goes to
$\infty$, or even for the precise number of election results
without a Condorcet winner, depending on the number $k$ of
voters.

It is not hard to see that the outcomes that have $A$ is the
Condorcet winner can be described by three homogeneous linear
inequalities $\lambda_i(x)>0$ whose coefficients are given in
Table \ref{ineq1} (relative to the lexicographic order of the
permutations of $A,B,C,D$).
\begin{table}[hbt]
{\small \tabcolsep=1.3pt
\begin{tabular}{rrrrrrrrrrrrrrrrrrrrrrrrr}
$\lambda_1$:&\phantom{-}1&\phantom{-}1&\phantom{-}1&\phantom{-}1&\phantom{-}1&\phantom{-}1&$-1$&$-1$&$-1$&$-1$&$-1$&$-1$&1&1&$-1$&$-1$&1&$-1$&1&1&$-1$&$-1$&1&$-1$\\
$\lambda_2$:&1&1&1&1&1&1&1&1&$-1$&$-1$&1&$-1$&$-1$&$-1$&$-1$&$-1$&$-1$&$-1$&1&1&1&$-1$&$-1$&$-1$\\
$\lambda_3$:&\rule[-1.5ex]{0ex}{2ex}1&1&1&1&1&1&1&1&1&$-1$&$-1$&$-1$&1&1&1&$-1$&$-1$&$-1$&$-1$&$-1$&$-1$&$-1$&$-1$&$-1$\\
\end{tabular}}
\vspace*{2ex} \caption{Inequalities expressing that $A$ beats
the other $3$ candidates}\label{ineq1}
\end{table}
They cut out a rational polytope from $U^{(n)}$, and the
probability of Condorcet's paradox can be computed from the
volume of the polytope. Finding the precise number of election
results without (or with) a Condorcet winner requires the
computation of the Ehrhart function of the semi-open polytope .
Neither Normaliz nor NmzIntegrate can yet compute Ehrhart
series for semi-open polytopes directly, but it is always
possible to fall back on inclusion/exclusion.

We refer the reader to \cite{BIS} for a description of problems
(2) and (3) and for the systems of linear inequalities to be
solved in each case. Normaliz 2.8 can indeed compute the
volumes and the Ehrhart series in dimension $24$ that arise
from tasks (1), (2) and (3) despite the fact that the
triangulations to be evaluated for (2) and (3) are formidable
(see Table \ref{ori} or \cite{BIS}).

As Schürmann \cite{Sch} observed, the computations can be
considerably simplified by exploiting the symmetries in the
inequalities: some variables share the same coefficients in
each inequality, for example the first $6$ variables in Table
\ref{ineq1}. Therefore they can be replaced by their sum, and
the replacement constitutes a  projection of  the original
polytopes, monoids or cones onto objects of smaller dimension.
For the Condorcet paradox the system of inequalities reduces to
Table \ref{ineqsymm}.
\begin{table}[hbt]
{\small \tabcolsep=1.3pt
\begin{tabular}{rrrrrrrrr}
1& -1&  1& 1&  1& -1& -1& -1\\
1&  1& -1& 1& -1&  1& -1& -1\\
1&  1& 1 &-1& -1& -1&  1& -1\\
\end{tabular}}
\vspace*{2ex} \caption{Inequalities exploiting the symmetries
in Table \ref{ineq1}}\label{ineqsymm}
\end{table}
However, instead of simply counting lattice points, one must
now count them with their numbers of preimages. These are given
by polynomials, namely products of binomial coefficients. In
our example the polynomial is
$$
\binom{y_1+5}{5}(y_2+1)(y_3+1)(y_4+1)(y_5+1)(y_6+1)(y_7+1)\binom{y_8+5}{5}
$$
where $y_1=x_1+\dots+x_6$ etc.In other words, the Ehrhart
function (or the volume) of a high dimensional polytope is
replaced by a generalized Ehrhart function of a polytope of
much lower dimension (or the virtual leading coefficient of the
quasipolynomial).

A priori it may not be clear that the replacement of
combinatorial complexity in high dimension by multivariate
polynomial arithmetic in low dimension pays dividends, but this
is indeed the case. Tables \ref{ori} and \ref{proj} compare
both approaches. The computations were run on a SUN xFire 4450
with $20$ parallel threads. If the computations in Table
\ref{ori} are restricted to volumes, they become faster by a
factor of approximately $3$.

\begin{table}[hbt]
\begin{tabular}{|l|r|r|}\hline
\rule[-0.1ex]{0ex}{2.5ex}computation&triangulation size&real time\\ \hline
\rule{0ex}{2.5ex}Condorcet paradox&1,473,107&00:00:30 h\\ \hline
\rule{0ex}{2.5ex}Condorcet efficiency&$347,225,775,338$&218:13:55 h\\ \hline
\rule{0ex}{2.5ex}plurality vs. cutoff&$257,744,341$,008&175:11:26 h\\ \hline
\end{tabular}
\vspace*{2ex} \caption{Computation times (real) for Ehrhart
series in dimension $24$}\label{ori}
\end{table}

\begin{table}[hbt]
\begin{tabular}{|l|r|r|r|r|r|r|}\hline
\rule[-0.1ex]{0ex}{2.5ex}computation&rank&$\deg f$&triangula-&Normaliz&gen Ehrhart&lead coeff\\
\rule[-0.1ex]{0ex}{2.5ex}&&&tion size&time&series time&time\\ \hline
\rule{0ex}{2.5ex}Condorcet paradox&8&16&17&0.05 sec&5.2 sec&0.08 sec\\ \hline
\rule{0ex}{2.5ex}Condorcet efficiency&13&11&17,953&0.41 sec&5:49:29 h&1:54:35 h\\ \hline
\rule{0ex}{2.5ex}plurality vs. cutoff&6&18&3&0.06 sec&18.4 sec&0.54 sec\\ \hline
\end{tabular}
\vspace*{2ex} \caption{Computation times (real) for symmetrized
data}\label{proj}
\end{table}

A welcome side effect of the computations of the generalized
Ehrhart functions is that they have confirmed the results
reported on in \cite{BIS}.

\end{document}